\documentclass[12pt, twoside, leqno]{article}
\usepackage{hyperref}
\hypersetup{
    colorlinks=true,       
    linkcolor=blue,        
    citecolor=blue,        
    urlcolor=blue,         
    pdftitle={Multiple Mertens theorems for arithmetic progressions}, 
    pdfauthor={Zhen Chen and Jun Rong Luo}   
}
\usepackage{amsmath,amsthm}
\usepackage{amssymb}
\usepackage{enumitem}
\usepackage{graphicx}
\usepackage[T1]{fontenc} 

\usepackage{longtable}
\usepackage[english]{babel}
\usepackage{caption}
\usepackage{mathrsfs}
\usepackage{amsfonts}

\newtheorem{theorem}{Theorem}[section]

\newtheorem{lemma}[theorem]{Lemma}
\newtheorem{proposition}[theorem]{Proposition}

\theoremstyle{definition}
\newtheorem{definition}[theorem]{Definition}
\newtheorem{remark}[theorem]{Remark}
\newtheorem{example}[theorem]{Example}

\newenvironment{thm}{\begin{theorem}}{\end{theorem}}
\newenvironment{prop}{\begin{proposition}}{\end{proposition}}
\newenvironment{rem}{\begin{remark}}{\end{remark}}

\numberwithin{equation}{section}

\begin{document}

\baselineskip=17pt

\title{Multiple Mertens theorems for arithmetic progressions}

\author{Zhen Chen\\
School of Mathematics\\
South China University of Technology\\
Guangzhou 510640, Guangdong, China\\
E-mail: 759930269@qq.com
\and
Junrong Luo\\
School of Mathematics and Statistics\\
Wuhan University\\
Wuhan 430072, Hubei, China\\
E-mail: junrong\_luo@whu.edu.cn} 

\date{}
\maketitle

\renewcommand{\thefootnote}{}
\renewcommand{\thefootnote}{\arabic{footnote}}
\setcounter{footnote}{0}

\begin{abstract}
We establish asymptotic formulas for sums of reciprocals of primes in arithmetic progressions, generalizing recent results on multiple Mertens evaluations by Tenenbaum, Qi, and Hu. Specifically, for any fixed constant $K>0$, we derive asymptotic expansions for the sums
$
\sum_{\substack{p_1\cdots p_n\leq x \\ p_i\equiv h_i \pmod{m_i} \\ i=1,\dots, n}}\frac{1}{p_1\cdots p_n} 
$ 
and the corresponding log-weighted sums. A key feature of our results is that the error terms hold \emph{uniformly} for moduli satisfying $m_i \le (\log x)^K$, a range accessible via the Siegel-Walfisz theorem. Furthermore, we identify the coefficients of the asymptotic expansion with the Taylor series of the reciprocal Gamma function, $1/\Gamma(z)$, providing a structural explanation for the lower-order terms.
\end{abstract}
\section{Introduction}

The classical Mertens theorems describe the asymptotic behavior of the sums of reciprocals of primes. In 1874, Mertens established the following asymptotic formulas (see \cite[Theorem 4.10 and Theorem 4.12]{ref1}, \cite{ref11}):
\begin{equation}\label{eq.mertens1}
    \sum_{p\leq x}\frac{\log p}{p}=\log x+O(1),
\end{equation}
and
\begin{equation}\label{eq.mertens2}
    \sum_{p\leq x}\frac{1}{p}=\log(\log x)+A+O\left(\frac{1}{\log x}\right),
\end{equation}
where $A$ is the Mertens constant. These estimates are fundamental to analytic number theory and play a crucial role in the elementary proof of the Prime Number Theorem.

Recently, the study of the distribution of products of primes has garnered significant attention. Early investigations into sums over products of two primes were conducted by Nathanson \cite{ref12}, who derived asymptotic formulas for the case $n=2$. Generalizing these results, Popa \cite{ref15,ref16} obtained asymptotic evaluations for sums over products of three primes using the Dirichlet hyperbola method. More generally, Tenenbaum \cite{ref27,ref23} applied the Selberg-Delange method to derive asymptotic formulas for $$\sum_{p_1\cdots p_k \leq x} (p_1\cdots p_k)^{-1}$$ for any $k \ge 1$, expressing the main terms via polynomials of $\log(\log x)$. Subsequently, Qi and Hu \cite{ref17} recovered and extended these results using elementary methods. We also mention that related estimates for sums over 'almost primes' and generalized constants have been investigated by Saidak \cite{ref21}, Wolf \cite{ref24}, and more recently by Bayless, Kinlaw, and Lichtman \cite{ref3,ref8,ref9}.

While generalized Mertens theorems provide a comprehensive picture for sums over all primes, the restriction to arithmetic progressions introduces non-trivial analytic challenges, particularly regarding uniformity. A direct application of elementary methods (as in \cite{ref17}) typically yields error terms that degrade rapidly as the modulus $m$ increases. To obtain estimates that remain effective uniformly for moduli $m \le (\log x)^K$, one must rigorously control the contribution of non-principal characters and the distribution of zeros of Dirichlet $L$-functions. This necessitates the use of the Siegel-Walfisz theorem to handle the potential existence of exceptional zeros, a subtlety absent in the case of unrestricted prime sums.

In this paper, we extend the multiple Mertens evaluations to the case of arithmetic progressions. By combining the elementary Dirichlet hyperbola method with the orthogonality of Dirichlet characters, we establish asymptotic formulas for multiple prime sums in arithmetic progressions. Crucially, we employ the classical Siegel-Walfisz theorem to control the contribution of non-principal characters, ensuring that our results hold uniformly for moduli growing polylogarithmically with $x$.

Our main result is the following theorem, which generalizes the results of \cite{ref17, ref23} to arithmetic progressions with uniform error terms.

\begin{thm}\label{thm1.4}
    Let $K > 0$ be an arbitrary fixed constant. Suppose $h_i, m_i$ are integers such that $(h_i,m_i)=1$ and $m_i \le (\log x)^K$ for all $i=1, 2,\dots, n$. Then we have 
    \begin{equation}
        \begin{aligned}\label{eq2}
            \sum\limits_{\substack{p_1\cdots p_n\le x
                    \\p_i\equiv h_i \pmod{m_i}}}
            \frac{\prod_{i=1}^{n}\varphi(m_i)}{p_1\cdots p_n}
            =\tau_n^{(n)}(x)+\sum\limits_{j=0}^{n-2}a_{n-j}\tau_{j}^{(n)}(x)+O\left(\frac{\log^{n-1}(\log x)}{\log x}\right).
        \end{aligned} 
    \end{equation}
    Here, the coefficients $\tau_j^{(n)}(x)$ are determined by the polynomial expansion
    \[ \prod_{i=1}^{n}(X+y_i)=\sum_{i=0}^{n}\tau_i^{(n)}(x)X^{n-i}, \]
    where $y_i=\log(\log x)+\varphi(m_i)A_{h_i,m_i}$, and $A_{h,m}$ denotes the Mertens constant associated with the arithmetic progression $h \pmod m$ (as defined in Theorem \ref{thm1.1}). The sequence $a_k$ is defined recursively by  $a_0 = 1$, $a_1 = 0$, $a_2=-\zeta(2)$, $a_3=2\zeta(3)$, and for $k\ge4$:
    \[
    a_k=(-1)^{k-1}(k-1)!\zeta(k)+\sum\limits_{i=1}^{k-3}(-1)^i \binom{k-1}{i} i!\zeta(i+1)a_{k-1-i}.
    \]
    The implied constant in the $O$-term depends only on $n$ and $K$.
\end{thm}

\begin{rem}\label{remark1}
    \textbf{(Uniformity and Siegel Zeros)} The condition $m_i \le (\log x)^K$ is imposed to utilize the Siegel-Walfisz theorem to obtain uniform estimates, which guarantees that the contribution from non-principal characters is negligible (see \cite{ref26} for a comprehensive discussion on exceptional zeros). If one assumes the Generalized Riemann Hypothesis (GRH), this range of uniformity can be significantly extended to $m_i \le x^{1/2-\epsilon}$. It is worth noting that while the Bombieri-Vinogradov theorem (see \cite[Chapter 28]{ref4}) establishes equidistribution for moduli up to $x^{1/2-\epsilon}$ in an \emph{average} sense, our evaluations require \emph{pointwise} estimates for specific sequences of moduli. Consequently, unconditionally, the current result represents the standard range of uniformity for such multiple prime sums.
\end{rem}
Theorem \ref{thm1.4} in the case of $n=1$ corresponds to the classical result for arithmetic progressions (see \cite{ref4}):
\begin{thm}\label{thm1.1}
    Suppose $K > 0$ is fixed. For $m \le (\log x)^K$ and $(h,m)=1$, we have
    \begin{align*}
        \sum\limits_{\substack{p\le x \\ p\equiv h \pmod m}}\frac{1}{p}
        =\frac{\log(\log x)}{\varphi(m)}+A_{h,m}
        +O\left(\frac{1}{\log x}\right).        
    \end{align*}
\end{thm}

Additionally, as an intermediate step, we derive the following estimate for the log-weighted sum:
\begin{thm}\label{thm1.3}
    Suppose $h_i, m_i \in \mathbb{N}$ with $(h_i, m_i)=1$ for $i \in \{1, \dots, n-1\}$. Let $y = \log(\log x)$. We define the \emph{associated polynomial} $\mathcal{Q}_{n-1}(Y)$ of degree $n-1$ as the convolution of the arithmetic progression coefficients $\mu_j^{(n-1)}$ (see Definition \ref{rek1}) and the coefficients $a_t$ defined in Theorem \ref{thm1.4}:
    \begin{equation}\label{eq:poly_def}
        \mathcal{Q}_{n-1}(Y) = \sum_{j=0}^{n-1} \mu_j^{(n-1)} \sum_{t=0}^{n-1-j} \binom{n-1-j}{t} a_t Y^{n-1-j-t}.
    \end{equation}
    Then for every natural number $k$, we have the asymptotic expansion:
    \begin{equation}\label{eqthm1.3}
        \begin{aligned}
             \sum_{\substack{p_1 \cdots p_{n-1} \le x \\ p_i \equiv h_i \pmod{m_i}}} & \frac{\prod_{i=1}^{n-1}\varphi(m_i) \log^k(p_1 \cdots p_{n-1})}{p_1 \cdots p_{n-1}} \\
            &= \log^k x \sum_{h=1}^{n-1} \frac{(-1)^{h-1}}{k^h} \mathcal{Q}_{n-1}^{(h)}(y) + O\left( \log^{k-1} x (\log \log x)^{n-1} \right),
        \end{aligned}
    \end{equation}
    where $\mathcal{Q}_{n-1}^{(h)}(y)$ denotes the $h$-th derivative of $\mathcal{Q}_{n-1}(Y)$ at $Y=y$.
\end{thm}
Finally, Section 4 provides an analytic perspective on the coefficients emerging from our inductive arguments. We identify the exponential generating function of the recursive sequence $a_k$ with the modified reciprocal Gamma function, $e^{-\gamma z}/\Gamma(1+z)$, which clarifies the structural origin of the lower-order terms. Furthermore, we demonstrate that the symmetric polynomials governing the main terms, $\tau_n^{(n)}(x)$, arise naturally from the singularity structure of the associated Dirichlet series at $s=1$. These analyses elucidate the consistency between our elementary evaluations and the general predictions of the Selberg-Delange method \cite{ref_delange, ref_selberg}.

\section{Preliminaries}

In this section, we gather the necessary tools and auxiliary results required for the proof of Theorem \ref{thm1.4}. We begin with general summation methods and then proceed to specific estimates involving Dirichlet characters and polylogarithms.

\subsection{General Summation Methods}

We first recall the multidimensional Dirichlet hyperbola method, which allows us to decompose multiple sums into manageable parts.

\begin{lemma}[Multiple Dirichlet Hyperbola Method, see {\cite[Proposition 2.1]{ref17}}]\label{lem1}
    Let $f_1, \cdots, f_k: \mathbb{N} \to \mathbb{C}$ be arithmetic functions. Define the summatory functions $S_{f_k}(x) = \sum_{n \le x} f_k(n)$ and $$S_{f_1, \cdots, f_{k-1}}(x) = \sum_{n_1 \cdots n_{k-1} \le x} f_1(n_1) \cdots f_{k-1}(n_{k-1}).$$ Then for any $0 < y < x$, the following identity holds:
    \begin{align*}
        \sum_{n_1 \cdots n_k \le x} f_1(n_1) \cdots f_k(n_k)
        &= \sum_{n_k \le y} f_k(n_k) S_{f_1, \cdots, f_{k-1}}\left(\frac{x}{n_k}\right) \\
        &\quad + \sum_{n_1 \cdots n_{k-1} \le \frac{x}{y}} f_1(n_1) \cdots f_{k-1}(n_{k-1}) S_{f_k}\left(\frac{x}{n_1 \cdots n_{k-1}}\right) \\
        &\quad - S_{f_k}(y) S_{f_1, \cdots, f_{k-1}}\left(\frac{x}{y}\right).
    \end{align*}
\end{lemma}

Next, we state the multiple Abel summation formula, which is essential for handling sums with logarithmic weights.

\begin{lemma}[Abel's Identity, see {\cite[Theorem 10.17]{ref5}}]\label{lem2}
    Let $a(n)$ be an arithmetic function and let $A(x) = \sum_{n \le x} a(n)$, with $A(x) = 0$ for $x < 1$. If $f$ has a continuous derivative on the interval $[y, x]$ where $0 < y < x$, then
    \begin{align*}
        \sum_{y < n \le x} a(n) f(n) = A(x) f(x) - A(y) f(y) - \int_{y}^{x} A(t) f'(t) \, dt.
    \end{align*}
\end{lemma}

\subsection{Estimates Involving Dirichlet Characters}

A critical component of our proof is controlling the error terms arising from sums over arithmetic progressions. To ensure our estimates hold uniformly for moduli growing with $x$, we employ the classical Siegel-Walfisz theorem.

\begin{lemma}[Siegel-Walfisz Theorem, see {\cite[Chapter 11]{ref4}, \cite[Chapter 5]{ref26}}]\label{lem:sw}
    Let $K > 0$ be an arbitrary fixed constant. For any modulus $m \le (\log x)^K$ and any non-principal Dirichlet character $\chi \pmod{m}$, we have the estimate
    \begin{equation*}
        \psi(x, \chi) = \sum_{n \le x} \Lambda(n) \chi(n) = O_{K}\left(x \exp\left(-c \sqrt{\log x}\right)\right),
    \end{equation*}
    where $c > 0$ is an absolute constant and the implied constant depends only on $K$.
\end{lemma}

Using Lemma \ref{lem:sw}, we can derive the following estimate for prime reciprocal sums twisted by non-principal characters.

\begin{lemma}\label{lem5}
    Let $K > 0$ be a fixed constant. Suppose $m \le (\log x)^K$ and $\chi$ is any non-principal Dirichlet character $\pmod{m}$. For any natural number $k \ge 1$, we have:
    \begin{align*}
        \sum_{p \le x} \frac{\chi(p) \log^k p}{p} = O_K(\log^{k-1} x).
    \end{align*}
\end{lemma}

\begin{proof}
  We split the proof into two cases.

\textbf{Case 1: $k=1$.}
  We first establish the bound for the base case. Let $\theta(x, \chi) = \sum_{p \le x} \chi(p) \log p$. Since $\psi(x, \chi) = \theta(x, \chi) + O(\sqrt{x} \log^2 x)$, Lemma \ref{lem:sw} implies that
 \[ \theta(x, \chi) \ll x \exp(-c_1 \sqrt{\log x}) \]
 uniformly for $m \le (\log x)^K$.
 Applying Abel's summation formula to the sum $\sum_{p \le x} \frac{\chi(p) \log p}{p}$, we obtain
 \begin{align*}
   \sum_{p \le x} \frac{\chi(p) \log p}{p} 
   &= \frac{\theta(x, \chi)}{x} + \int_{2}^{x} \frac{\theta(t, \chi)}{t^2} \, dt \\
   &\ll \exp(-c_1 \sqrt{\log x}) + \int_{2}^{x} \frac{\exp(-c_1 \sqrt{\log t})}{t} \, dt.
   \end{align*}
  The integral $\int_{2}^{\infty} \frac{\exp(-c_1 \sqrt{\log t})}{t} \, dt$ converges absolutely. Thus, the partial sums are uniformly bounded:
 \begin{equation}\label{eq:base_case}
     S(x) := \sum_{p \le x} \frac{\chi(p) \log p}{p} = O_K(1).
 \end{equation}

 \textbf{Case 2: $k \ge 2$.}
 We rewrite the sum as $\sum_{p \le x} \left( \frac{\chi(p) \log p}{p} \right) (\log p)^{k-1}$. Let $S(t)$ be defined as in \eqref{eq:base_case}. By Abel's identity (Lemma \ref{lem2}), we have
  \begin{align*}
   \sum_{p \le x} \frac{\chi(p) \log^k p}{p}
    &= S(x) (\log x)^{k-1} - \int_{2}^{x} S(t) \frac{d}{dt} \left( (\log t)^{k-1} \right) \, dt \\
   &= S(x) (\log x)^{k-1} - (k-1) \int_{2}^{x} S(t) (\log t)^{k-2} \frac{1}{t} \, dt.
 \end{align*}
 Since $S(t) = O_K(1)$ for all $t \ge 2$ and $\int_{2}^{x} \frac{(\log t)^{k-2}}{t} \, dt \ll (\log x)^{k-1}$, we deduce
 \begin{align*}
  \sum_{p \le x} \frac{\chi(p) \log^k p}{p} 
  &\ll (\log x)^{k-1} + (k-1) \int_{2}^{x} \frac{(\log t)^{k-2}}{t} \, dt \\
&= O_K(\log^{k-1} x) + O_K\left( (\log x)^{k-1} \right) \\
 &= O_K(\log^{k-1} x).
 \end{align*}
 This completes the proof.
\end{proof}

The next two lemmas handle the logarithmic weights appearing in the hyperbola method expansion.

\begin{lemma}\label{lem6}
    Let $\chi$ be a non-principal Dirichlet character $\pmod{m}$ with $m \le (\log x)^K$, where $K>0$ is a fixed constant. Then for any $k \in \mathbb{N}$,
    \begin{align*}
        \sum_{p \le \sqrt{x}} \frac{\chi(p)}{p} \left( \log \left( 1 - \frac{\log p}{\log x} \right) \right)^k = O\left(\frac{1}{\log x}\right).
    \end{align*}
\end{lemma}
\begin{proof}
    For $|y| < 1$, we have the Taylor expansion $\log^k(1-y) = (-1)^k \sum_{n=k}^{\infty} b_n y^n$, where the coefficients $b_n$ depend on $k$ and satisfy $b_n = O(n^k)$.
    Let $y = \frac{\log p}{\log x}$. For $p \le \sqrt{x}$, we have $y \le 1/2$. Thus, we can exchange the summation orders:
    \begin{align*}
        \sum_{p \le \sqrt{x}} \frac{\chi(p)}{p} \left( \log \left( 1 - \frac{\log p}{\log x} \right) \right)^k
        &= (-1)^k \sum_{p \le \sqrt{x}} \frac{\chi(p)}{p} \sum_{n=k}^{\infty} b_n \left( \frac{\log p}{\log x} \right)^n \\
        &= (-1)^k \sum_{n=k}^{\infty} \frac{b_n}{(\log x)^n} \sum_{p \le \sqrt{x}} \frac{\chi(p) \log^n p}{p}.
    \end{align*}
    By Lemma \ref{lem5}, the inner sum is $O(\log^{n-1} \sqrt{x}) = O( (\frac{1}{2} \log x)^{n-1} )$. Thus, the total sum is bounded by
    \[
        \ll \sum_{n=k}^{\infty} |b_n| \frac{(\log x)^{n-1}}{(\log x)^n} \frac{1}{2^{n-1}} = \frac{1}{\log x} \sum_{n=k}^{\infty} \frac{|b_n|}{2^{n-1}}.
    \]
    Since the series $\sum |b_n| 2^{-(n-1)}$ converges absolutely, the result follows.
\end{proof}

\begin{lemma}\label{lem7}
    Let $\chi_0$ be the principal character $\pmod{m}$ with $m \le (\log x)^K$, where $K>0$ is a fixed constant. Then for any $k \in \mathbb{N}$,
    \begin{align*}
        \sum_{p \le \sqrt{x}} \frac{\chi_0(p)}{p} \left( \log \left( 1 - \frac{\log p}{\log x} \right) \right)^k
        = \int_{0}^{1/2} \frac{\log^k(1-t)}{t} \, dt + O\left(\frac{1}{\log x}\right).
    \end{align*}
\end{lemma}

\begin{proof}
    Let $\Phi(p) = \left( \log \left( 1 - \frac{\log p}{\log x} \right) \right)^k$. Since $\chi_0(p) = 1$ for $p \nmid m$ and $0$ otherwise, we verify the decomposition:
    \begin{align*}
        \sum_{p \le \sqrt{x}} \frac{\chi_0(p)}{p} \Phi(p) 
        = \sum_{p \le \sqrt{x}} \frac{\Phi(p)}{p} - \sum_{\substack{p \le \sqrt{x} \\ p | m}} \frac{\Phi(p)}{p}.
    \end{align*}
    
    We first estimate the second sum, which represents the contribution from the prime factors of $m$. Since $m \le (\log x)^K$, for any $p|m$ we have the bound $\log p \le \log m \ll \log \log x$.
    Using the inequality $|\log(1-y)| \le 2|y|$ valid for small $y$, the weight function satisfies:
    \[
        |\Phi(p)| \ll \left( \frac{\log p}{\log x} \right)^k.
    \]
    Consequently, the sum over primes dividing $m$ is bounded by:
    \begin{align*}
        \left| \sum_{p|m} \frac{\Phi(p)}{p} \right| 
        &\ll \frac{1}{(\log x)^k} \sum_{p|m} \frac{\log^k p}{p} \\
        &\le \frac{(\log m)^{k-1}}{(\log x)^k} \sum_{p|m} \frac{\log p}{p}.
    \end{align*}
    Using the trivial bound $\sum_{p|m} \frac{\log p}{p} \le \sum_{p|m} \log p = \log m$, we obtain:
    \begin{equation}\label{eq:principal_error}
        \sum_{p|m} \frac{\Phi(p)}{p} \ll \left( \frac{\log m}{\log x} \right)^k \ll \left( \frac{K \log \log x}{\log x} \right)^k \ll \frac{1}{\log x}.
    \end{equation}
    Thus, the contribution from the non-coprime terms is negligible.
    
    For the main sum over all primes $p \le \sqrt{x}$, we apply the Prime Number Theorem in the form $\pi(t) = \operatorname{Li}(t) + E(t)$. Using partial summation (similar to the method in \cite[Proposition 3.2]{ref17}), the sum converges to the integral term:
    \[
        \sum_{p \le \sqrt{x}} \frac{\Phi(p)}{p} = \int_{0}^{1/2} \frac{\log^k(1-t)}{t} \, dt + O\left(\frac{1}{\log x}\right).
    \]
    Combining this with \eqref{eq:principal_error} completes the proof.
\end{proof}

Combining Lemma \ref{lem6} and Lemma \ref{lem7}, we obtain the following proposition which describes the asymptotic behavior of sums over arithmetic progressions involving logarithmic weights.

\begin{prop}\label{prop2}

    For every $k \in \mathbb{N}$ and integers $h, m$ with $(h, m) = 1$ and $m \le (\log x)^K$, where $K>0$ is a fixed constant, we have
 \begin{align*}
       \sum_{\substack{p \le \sqrt{x} \\ p \equiv h \pmod{m}}} \frac{\varphi(m)}{p} \left( \log \left( 1 - \frac{\log p}{\log x} \right) \right)^k
        = \int_{0}^{1/2} \frac{\log^k(1-t)}{t} \, dt + O\left(\frac{1}{\log x}\right).
    \end{align*}
\end{prop}

\begin{proof}
    Let $\Phi(p) = \left( \log \left( 1 - \frac{\log p}{\log x} \right) \right)^k$.
    By the orthogonality of Dirichlet characters, we write
    \begin{align*}
        \sum_{\substack{p \le \sqrt{x} \\ p \equiv h \pmod{m}}} \frac{\varphi(m)}{p} \Phi(p)
        &= \sum_{p \le \sqrt{x}} \frac{\Phi(p)}{p} \sum_{\chi \pmod{m}} \overline{\chi}(h) \chi(p) \\
        &= \sum_{\chi \pmod{m}} \overline{\chi}(h) \sum_{p \le \sqrt{x}} \frac{\chi(p)}{p} \Phi(p) \\
        &= \sum_{p \le \sqrt{x}} \frac{\chi_0(p)}{p} \Phi(p) + \sum_{\chi \neq \chi_0} \overline{\chi}(h) \sum_{p \le \sqrt{x}} \frac{\chi(p)}{p} \Phi(p).
    \end{align*}
    The result follows immediately by applying Lemma \ref{lem7} to the principal character term and Lemma \ref{lem6} to the sum over non-principal characters.
\end{proof}
\subsection{Combinatorial Identities}

In order to simplify the coefficients arising from the recursive process, we introduce the following notation regarding symmetric polynomials.

\begin{definition}\label{rek1}
    Let $\{i_1, i_2, \dots, i_j\}$ be a subset of $\{1, 2, \dots, n\}$. Let $y_r$ be defined as in Theorem \ref{thm1.4}, and let $C_r = \varphi(m_r)A_{h_r, m_r}$ denote the constant part of $y_r$.
    
    We define $\tau_k^{(j)}(x; i_1, \dots, i_j)$ as the coefficient of $X^{j-k}$ in the polynomial expansion:
    \begin{equation}
        \prod_{r=1}^{j}(X + y_{i_r}) = \sum_{k=0}^{j} \tau_k^{(j)}(x; i_1, \dots, i_j) X^{j-k}.
    \end{equation}
    In particular, the coefficients appearing in Theorem \ref{thm1.4} correspond to the full set of indices:
    \[ \tau_k^{(n)}(x) = \tau_k^{(n)}(x; 1, \dots, n). \]
    
   Similarly, we define the constant coefficients $\mu_k^{(n)}$ by the expansion:
    \begin{equation}
        \prod_{i=1}^{n}(X + C_i) = \sum_{k=0}^{n} \mu_k^{(n)} X^{n-k}.
    \end{equation}
    In other words, $\mu_k^{(n)}$ is the $k$-th elementary symmetric polynomial in the variables $\{C_1, \dots, C_n\}$.
\end{definition}

Using these definitions, we state the following combinatorial identities which are essential for the cancellation of cross-terms in the proof of the main theorem.

\begin{prop}\label{prop3}
    Let $\tau_k^{(j)}(x; i_1, \dots, i_j)$ and $\mu_k^{(j)}(i_1, \dots, i_j)$ be defined as in Definition \ref{rek1}. For $1 \le j \le n-1$, we have 
    $$\displaystyle \sum_{1 \le i_1 < \dots < i_j \le n-1} \sum_{k=1}^{j} \tau_{j-k}^{(j)}(x; i_1, \dots, i_j) = \sum_{k=1}^{j} \binom{n-1-j+k}{k} \tau_{j-k}^{(n-1)}(x).$$
\end{prop}

\begin{proof}
    Recall from Definition \ref{rek1} that $\tau_{j-k}^{(j)}(x; i_1, \dots, i_j)$ is the coefficient of $X^k$ in the expansion of $\prod_{r=1}^j (X+y_{i_r})$. By Vieta's formulas, this is exactly the elementary symmetric polynomial of degree $j-k$ in the variables $\{y_{i_1}, \dots, y_{i_j}\}$.
    
    It suffices to prove the identity term-wise for each fixed $k$. Let $d = j-k$. We consider the contribution of a specific monomial $Y = y_{t_1} y_{t_2} \dots y_{t_d}$, where $\{t_1, \dots, t_d\}$ is a subset of indices from $\{1, \dots, n-1\}$.
    
    The monomial $Y$ appears in the summand $\tau_{d}^{(j)}(x; i_1, \dots, i_j)$ if and only if the subset of indices $\{i_1, \dots, i_j\}$ contains the set $\{t_1, \dots, t_d\}$. Thus, we need to count the number of subsets of size $j$ from $\{1, \dots, n-1\}$ that contain these fixed $d$ elements.
    
    We must choose the remaining $j-d$ elements from the remaining $(n-1)-d$ available indices. The number of such choices is given by the binomial coefficient:
    \[
    \binom{(n-1)-d}{j-d} = \binom{(n-1)-(j-k)}{j-(j-k)} = \binom{n-1-j+k}{k}.
    \]
    By symmetry, every monomial of degree $j-k$ is counted exactly this many times in the sum over all subsets. Therefore, the total sum is equal to the elementary symmetric polynomial of degree $j-k$ over the full set $\{1, \dots, n-1\}$ (which is $\tau_{j-k}^{(n-1)}(x)$) multiplied by this combinatorial factor. Summing over all $k$ completes the proof.
\end{proof}
\begin{prop}\label{prop_relation}
    Let $y = \log(\log x)$. The coefficients $\tau_k^{(n)}(x)$ and $\mu_k^{(n)}$ satisfy the following polynomial relation for any $0 \le k \le n$:
    \begin{equation}\label{eq:tau_mu_identity}
        \tau_k^{(n)}(x) = \sum_{j=0}^{k} \binom{n-j}{k-j} \mu_j^{(n)} y^{k-j}.
    \end{equation}
    Consequently, $\tau_k^{(n)}(x)$ is a polynomial in $y$ of degree $k$ with leading coefficient $\binom{n}{k}$.
\end{prop}

\begin{proof}
    Recall the defining generating functions from Definition \ref{rek1}:
    \[
    \sum_{k=0}^{n} \tau_k^{(n)}(x) X^{n-k} = \prod_{i=1}^{n} (X + y_i) = \prod_{i=1}^{n} (X + y + C_i),
    \]
    and
    \[
    \sum_{j=0}^{n} \mu_j^{(n)} X^{n-j} = \prod_{i=1}^{n} (X + C_i).
    \]
    Let $Z = X + y$. The first equation can be rewritten as:
    \[
    \sum_{k=0}^{n} \tau_k^{(n)}(x) X^{n-k} = \prod_{i=1}^{n} (Z + C_i) = \sum_{j=0}^{n} \mu_j^{(n)} Z^{n-j} = \sum_{j=0}^{n} \mu_j^{(n)} (X + y)^{n-j}.
    \]
    Using the binomial expansion $(X+y)^{n-j} = \sum_{t=0}^{n-j} \binom{n-j}{t} X^{n-j-t} y^t$, we examine the coefficient of $X^{n-k}$ on the right-hand side.
    For a fixed $j$, the term $X^{n-k}$ appears when $n-j-t = n-k$, which implies $t = k-j$. Since $t \ge 0$, we must have $j \le k$.
    Summing over all valid $j$, the coefficient of $X^{n-k}$ is:
    \[
    \sum_{j=0}^{k} \mu_j^{(n)} \binom{n-j}{k-j} y^{k-j}.
    \]
    Comparing this with the left-hand side proves the identity.
\end{proof}

\subsection{Polylogarithmic Functions}
Finally, we recall properties of polylogarithmic functions which appear in the evaluation of the integrals. The polylogarithms $\operatorname{Li}_n(z)$ are defined recursively by
\begin{align*}
    \operatorname{Li}_1(z) &= -\log(1-z), \\
    \operatorname{Li}_n(z) &= \int_{0}^{z} \frac{\operatorname{Li}_{n-1}(t)}{t} \, dt = \sum_{k=1}^{\infty} \frac{z^k}{k^n}, \quad (n \ge 2).
\end{align*}
We require the following integral evaluation involving these functions.
\begin{prop}[See {\cite[Proposition 2.4]{ref17}}]\label{prop4}
    For any $m \in \mathbb{N}$,
    \begin{align*}
        \int_{0}^{1/2} \frac{\log^m(1-t)}{t} \, dt
        &= (-\log 2)^{m+1} + (-1)^m m! \zeta(m+1) \\
        &\quad + (-1)^{m-1} \sum_{s=1}^{m} \binom{m}{s} (\log 2)^{m-s} \operatorname{Li}_{s+1}\left(\frac{1}{2}\right).
    \end{align*}
\end{prop}

\section{Proof of Main Results}

We prove Theorem \ref{thm1.4} by induction on the number of prime variables $n$. The base case $n=1$ is established by Theorem \ref{thm1.1}.

Assume that Theorem \ref{thm1.4} holds for any number of variables $s \le n-1$. That is, for any integers $h_i, m_i$ with $(h_i, m_i)=1$ and $m_i \le (\log x)^K$, we have
\begin{equation}\label{eq3.1}
    \sum_{\substack{p_1\cdots p_s\le x \\ p_i\equiv h_i \pmod{m_i}}}
    \frac{\prod_{i=1}^{s}\varphi(m_i)}{p_1\cdots p_s}
    = \tau_s^{(s)}(x) + \sum_{j=0}^{s-2} a_{s-j} \tau_{j}^{(s)}(x) + O\left(\frac{\log^{s-1}(\log x)}{\log x}\right).
\end{equation}
We now show that the formula holds for $n$. Applying the Dirichlet hyperbola method (Lemma \ref{lem1}) with splitting parameter $y = \sqrt{x}$, we decompose the sum for $n$ variables as follows:
\begin{align*}
    \Sigma_n(x) &:= \sum_{\substack{p_1 \cdots p_n \le x \\ p_i \equiv h_i \pmod{m_i}}} \frac{\prod_{i=1}^{n}\varphi(m_i)}{p_1 \cdots p_n} \\
    &= \sum_{\substack{p_n \le \sqrt{x} \\ p_n \equiv h_n \pmod{m_n}}} \frac{\varphi(m_n)}{p_n}
    \sum_{\substack{p_1 \cdots p_{n-1} \le x/p_n \\ p_i \equiv h_i \pmod{m_i}}} \frac{\prod_{i=1}^{n-1}\varphi(m_i)}{p_1 \cdots p_{n-1}} \\
    &\quad + \sum_{\substack{p_1 \cdots p_{n-1} \le \sqrt{x} \\ p_i \equiv h_i \pmod{m_i}}} \frac{\prod_{i=1}^{n-1}\varphi(m_i)}{p_1 \cdots p_{n-1}}
    \sum_{\substack{p_n \le x/(p_1 \cdots p_{n-1}) \\ p_n \equiv h_n \pmod{m_n}}} \frac{\varphi(m_n)}{p_n} \\
    &\quad - \left( \sum_{\substack{p_1 \cdots p_{n-1} \le \sqrt{x} \\ p_i \equiv h_i \pmod{m_i}}} \frac{\prod_{i=1}^{n-1}\varphi(m_i)}{p_1 \cdots p_{n-1}} \right)
    \left( \sum_{\substack{p_n \le \sqrt{x} \\ p_n \equiv h_n \pmod{m_n}}} \frac{\varphi(m_n)}{p_n} \right).
\end{align*}
We denote these three terms by $\mathcal{A}, \mathcal{B},$ and $\mathcal{C}$ respectively.

\subsection{Estimation of $\mathcal{A}$}

Let $S_{n-1}(t)$ denote the sum over $n-1$ prime variables as defined in the induction hypothesis \eqref{eq3.1}:
\[
    S_{n-1}(t) = \sum_{\substack{p_1 \cdots p_{n-1} \le t \\ p_i \equiv h_i \pmod{m_i}}} \frac{\prod_{i=1}^{n-1}\varphi(m_i)}{p_1 \cdots p_{n-1}}.
\]
Then the term $\mathcal{A}$ is expressed as:
\[
    \mathcal{A} = \sum_{\substack{p_n \le \sqrt{x} \\ p_n \equiv h_n \pmod{m_n}}} \frac{\varphi(m_n)}{p_n} S_{n-1}\left(\frac{x}{p_n}\right).
\]

We apply the induction hypothesis \eqref{eq3.1} with the parameter $s=n-1$ and the argument $x/p_n$. The inner sum $S_{n-1}(x/p_n)$ is explicitly given by:
\begin{equation}\label{eq:inner_A}
    \tau_{n-1}^{(n-1)}\left(\frac{x}{p_n}\right) + \sum_{j=0}^{n-3} a_{n-1-j} \tau_{j}^{(n-1)}\left(\frac{x}{p_n}\right) + O\left(\frac{\log^{n-2}(\log x)}{\log x}\right).
\end{equation}

To evaluate the terms $\tau_{k}^{(n-1)}(x/p_n)$, recall that for the shifted argument, the variables satisfy $y_i(x/p_n) = y_i(x) + z$, where $z = \log(1 - \frac{\log p_n}{\log x})$.

By Vieta's formulas, $\tau_k^{(n-1)}(x/p_n)$ corresponds to the elementary symmetric polynomial of degree $k$ in the shifted variables $\{y_1+z, \dots, y_{n-1}+z\}$. This can be written as a sum over all subsets of indices $J \subseteq \{1, \dots, n-1\}$ of size $k$:
\[
\tau_k^{(n-1)}\left(\frac{x}{p_n}\right) = \sum_{\substack{J \subseteq \{1, \dots, n-1\} \\ |J| = k}} \prod_{j \in J} (y_j(x) + z).
\]
Expanding the product $\prod_{j \in J} (y_j + z)$ in powers of $z$, we have
\[
\prod_{j \in J} (y_j + z) = \sum_{i=0}^k z^i e_{k-i}(\{y_j\}_{j \in J}),
\]
where $e_{k-i}(\{y_j\}_{j \in J})$ denotes the elementary symmetric polynomial of degree $k-i$ restricted to the variables in the subset $J$. In the notation of Definition \ref{rek1}, this is precisely $\tau_{k-i}^{(k)}(x; J)$.
Substituting this back into the sum over $J$ and exchanging the order of summation, we obtain:
\[
\tau_k^{(n-1)}\left(\frac{x}{p_n}\right) = \sum_{i=0}^k z^i \left( \sum_{\substack{J \subseteq \{1, \dots, n-1\} \\ |J| = k}} \tau_{k-i}^{(k)}(x; J) \right).
\]
The inner sum is now exactly in the form required to apply Proposition \ref{prop3}. We apply Proposition \ref{prop3} with the subset size parameter $j=k$ and the coefficient degree parameter corresponding to $k-i$. The proposition states that this sum equals the global coefficient $\tau_{k-i}^{(n-1)}(x)$ multiplied by a combinatorial factor:
\[
\sum_{\substack{J \subseteq \{1, \dots, n-1\} \\ |J| = k}} \tau_{k-i}^{(k)}(x; J) = \binom{n-1-k+i}{i} \tau_{k-i}^{(n-1)}(x).
\]
Thus, we derive the exact translation formula:
\begin{equation}\label{eq:tau_shift}
    \tau_k^{(n-1)}\left(\frac{x}{p_n}\right) = \sum_{i=0}^{k} \binom{n-1-k+i}{i} \tau_{k-i}^{(n-1)}(x) z^i.
\end{equation}

Having established the expansion of the inner terms, we now substitute \eqref{eq:tau_shift} back into the expression for $\mathcal{A}$ and evaluate the outer sum over $p_n$. The term $\mathcal{A}$ becomes a linear combination of sums of the form:
\[
    M_i := \sum_{\substack{p_n \le \sqrt{x} \\ p_n \equiv h_n \pmod{m_n}}} \frac{\varphi(m_n)}{p_n} z^i = \sum_{p_n \le \sqrt{x}} \frac{\varphi(m_n)}{p_n} \left( \log \left( 1 - \frac{\log p_n}{\log x} \right) \right)^i.
\]
We distinguish two cases for the exponent $i$:
\begin{enumerate}
    \item \textbf{Case $i=0$:} The logarithmic weight is identically 1. Using Theorem \ref{thm1.1} with the upper bound $\sqrt{x}$, we have
    \[
    M_0 = y_n(\sqrt{x}) + O\left(\frac{1}{\log x}\right) = y_n(x) - \log 2 + O\left(\frac{1}{\log x}\right).
    \]
    \item \textbf{Case $i \ge 1$:} We apply Proposition \ref{prop2}. The sum converges to the integral moment:
    \[
    M_i = \int_{0}^{1/2} \frac{\log^i(1-t)}{t} \, dt + O\left(\frac{1}{\log x}\right) = \mathcal{I}_i + O\left(\frac{1}{\log x}\right).
    \]
\end{enumerate}

Now we collect the terms. The contribution from the leading term\\$\tau_{n-1}^{(n-1)}(x/p_n)$ is:
\[
    \sum_{p_n} \frac{\varphi(m_n)}{p_n} \sum_{i=0}^{n-1} \tau_{n-1-i}^{(n-1)}(x) z^i = \tau_{n-1}^{(n-1)}(x) M_0 + \sum_{i=1}^{n-1} \tau_{n-1-i}^{(n-1)}(x) M_i.
\]
Similarly, the contribution from the lower order terms\\$\sum_{j=0}^{n-3} a_{n-1-j} \tau_{j}^{(n-1)}(x/p_n)$ is:
\[
    \sum_{j=0}^{n-3} a_{n-1-j} \sum_{i=0}^j \binom{n-1-j+i}{i} \tau_{j-i}^{(n-1)}(x) M_i.
\]

Combining these parts and grouping by $M_0$ (the term $y_n - \log 2$) and $M_i$ (the integrals $\mathcal{I}_i$), we arrive at the final expansion for $\mathcal{A}$:
\begin{align}\label{eqA}
    \mathcal{A} &= \left( \tau_{n-1}^{(n-1)}(x) + \sum_{j=0}^{n-3} a_{n-1-j} \tau_{j}^{(n-1)}(x) \right) (y_n - \log 2) + \sum_{i=1}^{n-1} \tau_{n-1-i}^{(n-1)}(x) \mathcal{I}_i  \notag \\
    &\quad + \sum_{j=1}^{n-3} a_{n-1-j} \sum_{i=1}^{j} \binom{n-1-j+i}{i} \tau_{j-i}^{(n-1)}(x) \mathcal{I}_i + O\left(\frac{\log^{n-1}(\log x)}{\log x}\right).
\end{align}
\subsection{Proof of Theorem \ref{thm1.3}}

In this subsection, we establish the asymptotic formula for the log-weighted sums.

\begin{proof}[Proof of Theorem \ref{thm1.3}]
    Let $S_{n-1}(t) = \sum_{p_1 \dots p_{n-1} \le t} \frac{\prod \varphi(m_i)}{p_1 \dots p_{n-1}}$ denote the unweighted sum. We assume the validity of Theorem \ref{thm1.4} for $n-1$ variables. By the induction hypothesis, we have
    \[
    S_{n-1}(t) = \sum_{k=0}^{n-1} a_{n-1-k} \tau_{k}^{(n-1)}(t) + R(t),
    \]
    where the remainder satisfies $R(t) = O\left( \frac{(\log \log t)^{n-2}}{\log t} \right)$.

    Let $Y(t) = \log \log t$. To express the main term explicitly as a polynomial in $Y(t)$, we invoke Proposition \ref{prop_relation}. Applying identity \eqref{eq:tau_mu_identity} with the dimension parameter $n-1$, we have
    \[
        \tau_k^{(n-1)}(t) = \sum_{j=0}^k \binom{n-1-j}{k-j} \mu_j^{(n-1)} (Y(t))^{k-j}.
    \]
    Substituting this relation into the expression for $S(t)$ and interchanging the order of summation, we obtain:
    \begin{align*}
        \sum_{k=0}^{n-1} a_{n-1-k} \tau_{k}^{(n-1)}(t) 
        &= \sum_{k=0}^{n-1} a_{n-1-k} \sum_{j=0}^k \binom{n-1-j}{k-j} \mu_j^{(n-1)} (Y(t))^{k-j} \\
        &= \sum_{j=0}^{n-1} \mu_j^{(n-1)} \sum_{k=j}^{n-1} a_{n-1-k} \binom{n-1-j}{k-j} (Y(t))^{k-j}.
    \end{align*}
    Let $r = k-j$ denote the exponent of $Y(t)$. The inner sum becomes
    \[
    \sum_{r=0}^{n-1-j} a_{n-1-j-r} \binom{n-1-j}{r} (Y(t))^r.
    \]
    Letting $t = n-1-j-r$ (so that the index of $a$ is $t$), and noting that $\binom{N}{N-t} = \binom{N}{t}$, this inner sum corresponds exactly to the inner summation in the definition of $\mathcal{Q}_{n-1}(Y)$ given in \eqref{eq:poly_def}.
    Thus, we have rigorously established that
    \begin{equation}\label{eq:S_t_explicit}
        S(t) = \mathcal{Q}_{n-1}(\log \log t) + R(t).
    \end{equation}

    Next, we evaluate the weighted sum $\mathcal{W}(x) = \sum_{p_1 \dots p_{n-1} \le x} \frac{\prod \varphi(m_i) \log^k(p_1 \dots p_{n-1})}{p_1 \dots p_{n-1}}$ using Abel's summation formula:
    \[
    \mathcal{W}(x) = S_{n-1}(x) \log^k x - k \int_{2}^{x} S_{n-1}(t) (\log t)^{k-1} \frac{dt}{t}.
    \]
    Substituting \eqref{eq:S_t_explicit} into the integral, we split the calculation into a main term contribution $\mathcal{M}$ and an error term contribution $\mathcal{E}$.
    
    \textbf{1. Estimation of the Main Term $\mathcal{M}$:}
    Let $u = \log \log t$. The measure transforms as $dt/t = e^u du$. The integral becomes
    \[
    \mathcal{I} = k \int_{\log\log 2}^{\log \log x} \mathcal{Q}_{n-1}(u) e^{ku} \, du.
    \]
    We evaluate this integral using repeated integration by parts. For any polynomial $P(u)$, we have the identity $$\int P(u) e^{ku} du = e^{ku} \sum_{h \ge 0} (-1)^h k^{-(h+1)} P^{(h)}(u).$$
    Applying this to $\mathcal{Q}_{n-1}$ and evaluating at the limits $Y = \log \log x$ and $u_0 = \log \log 2$:
    \[
    \mathcal{I} = \left[ e^{ku} \sum_{h=0}^{n-1} \frac{(-1)^h}{k^h} \mathcal{Q}_{n-1}^{(h)}(u) \right]_{u_0}^{Y}.
    \]
    Combining this with the boundary term $$S_{n-1}(x)\log^k x = (\mathcal{Q}_{n-1}(Y) + R(x)) \log^k x,$$ the term corresponding to $h=0$ cancels exactly with the polynomial part of $S_{n-1}(x)$. The contribution from the lower limit $u_0$ is constant with respect to $x$, say $C(n,k) = O(1)$. Thus, the main part of $\mathcal{W}(x)$ is
    \begin{align*}
        \mathcal{M} &= \mathcal{Q}_{n-1}(Y)\log^k x - \left( \log^k x \sum_{h=0}^{n-1} \frac{(-1)^h}{k^h} \mathcal{Q}_{n-1}^{(h)}(Y) - C(n,k) \right) \\
        &= \log^k x \sum_{h=1}^{n-1} \frac{(-1)^{h-1}}{k^h} \mathcal{Q}_{n-1}^{(h)}(Y) + O(1).
    \end{align*}

    \textbf{2. Estimation of the Error Term $\mathcal{E}$:}
    The error contribution arises from the remainder $R(t)$ in both the boundary term and the integral.
    For the boundary term, we have $R(x)\log^k x \ll (\log x)^{k-1} (\log \log x)^{n-2}$.
    For the integral term, we must bound
    \[
    \mathcal{E}_{int} = k \int_{2}^{x} R(t) (\log t)^{k-1} \frac{dt}{t}.
    \]
    Using the bound $|R(t)| \le C \frac{(\log \log t)^{n-2}}{\log t}$ for $t \ge 3$, and changing variables to $v = \log t$, we have
    \[
    |\mathcal{E}_{int}| \ll \int_{\log 2}^{\log x} \frac{(\log v)^{n-2}}{v} v^{k-1} \, dv = \int_{\log 2}^{\log x} v^{k-2} (\log v)^{n-2} \, dv.
    \]
    We distinguish two cases based on the exponent of $v$:
    \begin{itemize}
        \item \textbf{Case $k=1$:} The integral becomes $\int_{\log 2}^{\log x} v^{-1} (\log v)^{n-2} \, dv$.
        Setting $w = \log v$, this is $$\int_{\log\log2}^{\log\log x} w^{n-2} dw \asymp (\log \log x)^{n-1}.$$
        Thus, the error is $O((\log \log x)^{n-1})$.
        
        \item \textbf{Case $k \ge 2$:} We employ integration by parts with $U(v) = (\log v)^{n-2}$ and $dV = v^{k-2} \, dv$. Then $dU = (n-2)(\log v)^{n-3} v^{-1} \, dv$ and $V = \frac{v^{k-1}}{k-1}$. We obtain
        \begin{align*}
            \int_{\log 2}^{\log x} v^{k-2} (\log v)^{n-2} \, dv &= \left[ \frac{v^{k-1}}{k-1} (\log v)^{n-2} \right]_{\log 2}^{\log x} \\
            &\quad - \frac{n-2}{k-1} \int_{\log 2}^{\log x} v^{k-2} (\log v)^{n-3} \, dv.
        \end{align*}
        The boundary term at $\log x$ contributes $O\left( (\log x)^{k-1} (\log \log x)^{n-2} \right)$. For the remaining integral, we note that $(\log v)^{n-3}$ is monotonic, so
        \begin{align*}
             \int_{\log 2}^{\log x} v^{k-2} (\log v)^{n-3} \, dv &\le (\log \log x)^{n-3} \int_{\log 2}^{\log x} v^{k-2} \, dv \\
             &\ll (\log x)^{k-1} (\log \log x)^{n-3}.
        \end{align*}
        Since $n$ is fixed, the second term is of strictly lower order than the boundary term. Thus, the total integral is dominated by the boundary term:
        \[
        \ll (\log x)^{k-1} (\log \log x)^{n-2}.
        \]
    \end{itemize}
    In both cases, the integral error term dominates the boundary error term $R(x)\log^k x$.
    Therefore, the total error is bounded by
    \[
    O\left( (\log x)^{k-1} (\log \log x)^{\max(n-1, n-2)} \right).
    \]
    For $k=1$, the power of $\log \log x$ is $n-1$. For $k \ge 2$, it is $n-2$ (which is subsumed by the stated error term). To provide a uniform statement, the error term $O(\log^{k-1} x (\log \log x)^{n-1})$ covers all cases. This completes the proof.
\end{proof}
   \subsection{Estimation of $\mathcal{B}$}

We now estimate the second term $\mathcal{B}$, defined as
\[
\mathcal{B} = \sum_{\substack{p_1 \cdots p_{n-1} \le \sqrt{x} \\ p_i \equiv h_i \pmod{m_i}}} \frac{\prod_{i=1}^{n-1}\varphi(m_i)}{p_1 \cdots p_{n-1}} \left( \sum_{\substack{p_n \le x/(p_1 \cdots p_{n-1}) \\ p_n \equiv h_n \pmod{m_n}}} \frac{\varphi(m_n)}{p_n} \right).
\]
Let $P = p_1 \cdots p_{n-1}$. We apply Theorem \ref{thm1.1} to the inner sum over $p_n$. Since the upper bound is $x/P$ and $P \le \sqrt{x}$, we have $\log(x/P) \asymp \log x$. Thus, the inner sum satisfies
\begin{align*}
    \sum_{\substack{p_n \le x/P \\ p_n \equiv h_n \pmod{m_n}}} \frac{\varphi(m_n)}{p_n} 
    &= y_n + \log\left(1 - \frac{\log P}{\log x}\right) + O\left(\frac{1}{\log x}\right),
\end{align*}
where $y_n = \log \log x + \varphi(m_n)A_{h_n, m_n}$.

Substituting this asymptotic formula back into the expression for $\mathcal{B}$, we decompose the total sum into two distinct parts: a linear term $\mathcal{B}_{lin}$ involving $y_n$, and a series term $\mathcal{B}_{ser}$ arising from the logarithmic expansion.
\[
    \mathcal{B} = \mathcal{B}_{lin} + \mathcal{B}_{ser} + O\left(\frac{(\log \log x)^{n-1}}{\log x}\right).
\]
We estimate these two terms separately.

\textbf{Step 1: Estimation of the linear term $\mathcal{B}_{lin}$.}

The term involving $y_n$ factors out of the summation over $P$. We can write
\[
    \mathcal{B}_{lin} = y_n \sum_{\substack{P \le \sqrt{x} \\ p_i \equiv h_i \pmod{m_i}}} \frac{\prod_{i=1}^{n-1}\varphi(m_i)}{P} = y_n S(\sqrt{x}),
\]
where $S(\sqrt{x})$ is the sum over $n-1$ variables. We invoke the induction hypothesis \eqref{eq3.1} (with $x$ replaced by $\sqrt{x}$) to estimate $S(\sqrt{x})$. This yields
\begin{equation}\label{eq:B_lin_est}
    \mathcal{B}_{lin} = y_n \left( \tau_{n-1}^{(n-1)}(\sqrt{x}) + \sum_{j=0}^{n-3} a_{n-1-j} \tau_{j}^{(n-1)}(\sqrt{x}) \right) + O\left(\frac{ (\log \log x)^{n-1}}{\log x}\right).
\end{equation}
The error term here is subsumed by the total error term.

\textbf{Step 2: Estimation of the series term $\mathcal{B}_{ser}$.}

We expand the logarithmic term into its power series representation:
\[
\log\left(1 - \frac{\log P}{\log x}\right) = -\sum_{k=1}^\infty \frac{1}{k} \left(\frac{\log P}{\log x}\right)^k.
\]
Substituting this series into the definition of $\mathcal{B}_{ser}$, we obtain
\[
    \mathcal{B}_{ser} = -\sum_{k=1}^\infty \frac{1}{k (\log x)^k} \sum_{\substack{P \le \sqrt{x} \\ p_i \equiv h_i \pmod{m_i}}} \frac{\prod_{i=1}^{n-1} \varphi(m_i) \log^k P}{P}.
\]
To evaluate the inner sum over $P$, we apply Theorem \ref{thm1.3} with the upper bound $\sqrt{x}$. Note that the logarithmic variable in the asymptotic expansion becomes $\log \sqrt{x} = \frac{1}{2}\log x$. Consequently, the inner sum is
\[
    \left(\frac{1}{2}\log x\right)^k \sum_{h=1}^{n-1} \frac{(-1)^{h-1}}{k^h} \mathcal{Q}_{n-1}^{(h)}(y) + O\left( \log^{k-1} x (\log \log x)^{n-1} \right).
\]
Inserting this result back into the expression for $\mathcal{B}_{ser}$, the powers of $\log x$ cancel, yielding the factor $(1/2)^k$. We exchange the order of summation between $k$ and $h$:
\begin{align*}
    \mathcal{B}_{ser} &= - \sum_{k=1}^\infty \frac{1}{k} \left(\frac{1}{2}\right)^k \left( \sum_{h=1}^{n-1} \frac{(-1)^{h-1}}{k^h} \mathcal{Q}_{n-1}^{(h)}(y)+ O\left( \log^{-1} x (\log \log x)^{n-1} \right) \right) \\
    &= \left(\sum_{h=1}^{n-1} (-1)^h \mathcal{Q}_{n-1}^{(h)}(y)+ O\left( \log^{k-1} x (\log \log x)^{n-1} \right)+\right) \left( \sum_{k=1}^\infty \frac{(1/2)^k}{k^{h+1}} \right).
\end{align*}
We recognize the inner summation over $k$ as the polylogarithm $\operatorname{Li}_{h+1}(1/2)$. Thus,
\begin{equation}\label{eq:B_ser_est}
    \mathcal{B}_{ser} = \sum_{h=1}^{n-1} (-1)^h \operatorname{Li}_{h+1}\left(\frac{1}{2}\right) \mathcal{Q}_{n-1}^{(h)}(y)+O\left( \log^{-1} x (\log \log x)^{n-1}\right).
\end{equation}

Finally, summing the contributions from \eqref{eq:B_lin_est} and \eqref{eq:B_ser_est}, we arrive at the asymptotic formula for $\mathcal{B}$:
\begin{equation}\label{eqB}
    \begin{aligned}
        \mathcal{B} &= \left( \tau_{n-1}^{(n-1)}(\sqrt{x}) + \sum_{j=0}^{n-3} a_{n-1-j} \tau_{j}^{(n-1)}(\sqrt{x}) \right) y_n \\
        &\quad + \sum_{h=1}^{n-1} (-1)^h \operatorname{Li}_{h+1}\left(\frac{1}{2}\right) \mathcal{Q}_{n-1}^{(h)}(y)
        + O\left(\frac{\log^{n-1}(\log x)}{\log x}\right).
    \end{aligned}
\end{equation}
This completes the estimation of $\mathcal{B}$.

\subsection{Estimation of $\mathcal{C}$}

The term $\mathcal{C}$ is the product of the two marginal sums evaluated at $\sqrt{x}$:
\[
\mathcal{C} = \left( \sum_{\substack{p_1 \dots p_{n-1} \le \sqrt{x} \\ p_i \equiv h_i \pmod{m_i}}} \frac{\prod_{i=1}^{n-1}\varphi(m_i)}{p_1 \dots p_{n-1}} \right)
\left( \sum_{\substack{p_n \le \sqrt{x} \\ p_n \equiv h_n \pmod{m_n}}} \frac{\varphi(m_n)}{p_n} \right).
\]
Let $S_{n-1}(\sqrt{x})$ and $S_1(\sqrt{x})$ denote the first and second factors, respectively.

First, we estimate $S_1(\sqrt{x})$. Applying Theorem \ref{thm1.1} with the upper bound $\sqrt{x}$, we have
\[
S_1(\sqrt{x}) = y_n(\sqrt{x}) + O\left(\frac{1}{\log x}\right).
\]
Recalling the definition $y_n(x) = \log \log x + \varphi(m_n)A_{h_n, m_n}$, we observe that
\[
y_n(\sqrt{x}) = \log(\log \sqrt{x}) + \varphi(m_n)A_{h_n, m_n} = y_n(x) - \log 2.
\]
Thus,
\begin{equation}\label{eq:S1_est}
    S_1(\sqrt{x}) = y_n - \log 2 + O\left(\frac{1}{\log x}\right).
\end{equation}

Next, we estimate $S_{n-1}(\sqrt{x})$. Using the induction hypothesis \eqref{eq3.1} for $n-1$ variables at $\sqrt{x}$, we write
\begin{equation}\label{eq:Sn-1_est}
    S_{n-1}(\sqrt{x}) = \mathcal{M}_{n-1}(\sqrt{x}) + \mathcal{E}_{n-1}(\sqrt{x}),
\end{equation}
where the main term is $\mathcal{M}_{n-1}(\sqrt{x}) = \tau_{n-1}^{(n-1)}(\sqrt{x}) + \sum_{j=0}^{n-3} a_{n-1-j} \tau_{j}^{(n-1)}(\sqrt{x})$, and the error term satisfies $\mathcal{E}_{n-1}(\sqrt{x}) = O\left( \frac{(\log \log x)^{n-2}}{\log x} \right)$.

Multiplying \eqref{eq:S1_est} and \eqref{eq:Sn-1_est}, we obtain
\begin{align*}
    \mathcal{C} &= \left( \mathcal{M}_{n-1}(\sqrt{x}) + \mathcal{E}_{n-1}(\sqrt{x}) \right) \left( y_n - \log 2 + O\left(\frac{1}{\log x}\right) \right) \\
    &= \mathcal{M}_{n-1}(\sqrt{x})(y_n - \log 2) + \mathcal{R}_{\mathcal{C}},
\end{align*}
where the remainder $\mathcal{R}_{\mathcal{C}}$ collects the cross-terms:
\[
\mathcal{R}_{\mathcal{C}} = \mathcal{M}_{n-1}(\sqrt{x}) \cdot O\left(\frac{1}{\log x}\right) + \mathcal{E}_{n-1}(\sqrt{x}) \cdot (y_n - \log 2) + \mathcal{E}_{n-1}(\sqrt{x}) \cdot O\left(\frac{1}{\log x}\right).
\]
Since $\mathcal{M}_{n-1}(\sqrt{x}) \asymp (\log \log x)^{n-1}$ and $y_n \asymp \log \log x$, the dominant error contribution comes from the first term:
\[
\mathcal{R}_{\mathcal{C}} \ll \frac{(\log \log x)^{n-1}}{\log x}.
\]
This allows us to state the asymptotic formula for $\mathcal{C}$ as:
\begin{equation}\label{eqC}
    \mathcal{C} = \left( \tau_{n-1}^{(n-1)}(\sqrt{x}) + \sum_{j=0}^{n-3} a_{n-1-j} \tau_{j}^{(n-1)}(\sqrt{x}) \right) (y_n - \log 2) + O\left(\frac{\log^{n-1}(\log x)}{\log x}\right).
\end{equation}
(Note: The polynomials $\tau_j^{(n-1)}(\sqrt{x})$ are evaluated at arguments shifted by $-\log 2$ relative to $\tau_j^{(n-1)}(x)$, which will be handled in the subsequent combination step).

\subsection{Conclusion: Synthesis and Coefficient Verification}

We combine the estimates derived in the previous subsections to form the total sum $\Sigma_n(x) = \mathcal{A} + \mathcal{B} - \mathcal{C}$.

First, we explicitly expand the difference $\mathcal{B} - \mathcal{C}$.
From \eqref{eqB}, the linear component of $\mathcal{B}$ is $y_n \sum_{j=0}^{n-1} a_{n-1-j} \tau_{j}^{(n-1)}(\sqrt{x})$.
From \eqref{eqC}, the main component of $\mathcal{C}$ is $(y_n - \log 2) \sum_{j=0}^{n-1} a_{n-1-j} \tau_{j}^{(n-1)}(\sqrt{x})$.
Subtracting these, the terms involving $y_n$ cancel exactly:
\begin{align}\label{eq:diff_explicit}
    \mathcal{B} - \mathcal{C} &= \log 2 \sum_{j=0}^{n-1} a_{n-1-j} \tau_{j}^{(n-1)}(\sqrt{x}) \\
    &\quad + \sum_{h=1}^{n-1} (-1)^h \operatorname{Li}_{h+1}\left(\frac{1}{2}\right) \mathcal{Q}_{n-1}^{(h)}(y) + O\left(\frac{\log^{n-1}(\log x)}{\log x}\right). \notag
\end{align}

Now we add $\mathcal{A}$ (from \eqref{eqA}) to this expression. We decompose $\Sigma_n(x)$ into a part linear in $y_n$ and a constant part $\mathcal{K}(x)$ (with respect to $y_n$).

\textbf{1. The Linear Part $\mathcal{L}(x)$.}
The only term involving $y_n$ comes from the main part of $\mathcal{A}$. The $y_n$ terms in $\mathcal{B}$ and $\mathcal{C}$ have cancelled as shown above. Thus:
\[
    \mathcal{L}(x) = y_n \sum_{j=0}^{n-1} a_{n-1-j} \tau_{j}^{(n-1)}(x).
\]
Using the recurrence relation $\tau_k^{(n)}(x) = \tau_k^{(n-1)}(x) + y_n \tau_{k-1}^{(n-1)}(x)$, this term correctly generates the upper degree components of the target polynomials.

\textbf{2. The Constant Part $\mathcal{K}(x)$ and Verification of $a_k$.}
The remaining terms collect all contributions that do not multiply $y_n$.
From $\mathcal{A}$, we have the constant term $-\log 2 \sum_{j=0}^{n-1} a_{n-1-j} \tau_{j}^{(n-1)}(x)$ and the integral terms.
Combining this with the first term of \eqref{eq:diff_explicit}, we form the boundary difference:
\[
    -\log 2 \sum_{j=0}^{n-1} a_{n-1-j} \left( \tau_{j}^{(n-1)}(x) - \tau_{j}^{(n-1)}(\sqrt{x}) \right).
\]
Thus, the full expression for $\mathcal{K}(x)$ is:
\begin{align*}
    \mathcal{K}(x) &= \underbrace{\sum_{j=0}^{n-1} a_{n-1-j} \sum_{i=1}^{j+1} \binom{n-1-j+i}{i} \tau_{j+1-i}^{(n-1)}(x) \mathcal{I}_i}_{\text{Integral terms from } \mathcal{A}} \\
    &\quad \underbrace{- \log 2 \sum_{j=0}^{n-1} a_{n-1-j} \left( \tau_{j}^{(n-1)}(x) - \tau_{j}^{(n-1)}(\sqrt{x}) \right)}_{\text{Boundary terms from } \mathcal{A} \text{ and } \mathcal{B}-\mathcal{C}} \\
    &\quad + \underbrace{\sum_{h=1}^{n-1} (-1)^h \operatorname{Li}_{h+1}\left(\frac{1}{2}\right) \mathcal{Q}_{n-1}^{(h)}(y)}_{\text{Polylogarithm terms}}.
\end{align*}

We must prove that $\mathcal{K}(x) = \sum_{j=0}^{n-2} a_{n-j} \tau_j^{(n-1)}(x)$. To do this, we extract the total coefficient of the basis polynomial $\tau_{n-1-k}^{(n-1)}(x)$ for a fixed $k \ge 1$. Let $\Theta_k$ denote this coefficient.

We analyze the contribution to $\Theta_k$ from each of the three parts:

\begin{enumerate}
    \item \textbf{From Integrals:} The term $\tau_{n-1-k}^{(n-1)}(x)$ appears when the summation index satisfies $j+1-i = n-1-k$, i.e., $i = k - (n-1-j)$. Letting $r = k - (n-1-j)$, the coefficient is:
    \[ \sum_{r=1}^k \binom{k}{r} a_{k-r} \mathcal{I}_r. \]

    \item \textbf{From Boundary Terms:} We use the polynomial translation formula. Since $y(\sqrt{x}) = y(x) - \log 2$, we have:
    \[ \tau_j^{(n-1)}(\sqrt{x}) = \sum_{v=0}^j \binom{n-1-j+v}{v} \tau_{j-v}^{(n-1)}(x) (-\log 2)^v. \]
    Substituting this into the boundary sum, the term corresponding to $\tau_{n-1-k}^{(n-1)}(x)$ arises when $j-v = n-1-k$, i.e., $v = k - (n-1-j)$. The coefficient calculation involves the term $\log 2 \cdot (-\log 2)^r$. Since $\log 2 = -(-\log 2)$, this becomes $-(-\log 2)^{r+1}$. Thus, the coefficient is:
    \[ \sum_{r=1}^{k} \binom{k}{r} a_{k-r} \left[ -(-\log 2)^{r+1} \right] = - \sum_{r=1}^{k} \binom{k}{r} a_{k-r} (-\log 2)^{r+1}. \]
    (Note: The formal summation index $r$ starts from 1 because the $v=0$ term cancels in the difference).

    \item \textbf{From Polylogarithms:} The derivative $\mathcal{Q}_{n-1}^{(h)}(y)$ extracts coefficients $a_t$. For the specific term $\tau_{n-1-k}^{(n-1)}$, this yields $(-1)^k a_0 \operatorname{Li}_{k+1}(1/2)$.
\end{enumerate}

Summing these explicit contributions, the condition $\Theta_k = a_k$ becomes the following identity:
\begin{equation}\label{eq:ak_explicit}
    a_k = \sum_{r=1}^{k} \binom{k}{r} a_{k-r} \mathcal{I}_r - \sum_{r=1}^{k} \binom{k}{r} a_{k-r} (-\log 2)^{r+1} + (-1)^k a_0 \operatorname{Li}_{k+1}\left(\frac{1}{2}\right).
\end{equation}

We now verify that our definition of $a_k$ satisfies this.
Recall Proposition \ref{prop4} for $\mathcal{I}_r$. To ensure the cancellation holds, we note the expansion:
\[
    \mathcal{I}_r = (-\log 2)^{r+1} + (-1)^r r! \zeta(r+1) + (-1)^{r-1} \sum_{s=1}^{r} \binom{r}{s} (\log 2)^{r-s} \operatorname{Li}_{s+1}\left(\frac{1}{2}\right).
\]
Substitute this into the first sum of \eqref{eq:ak_explicit}:
\begin{itemize}
    \item The terms $(-\log 2)^{r+1}$ from $\mathcal{I}_r$ (which are positive in the integral expansion) exactly cancel the boundary terms $-(-\log 2)^{r+1}$ (which are negative) in the second sum.
    \item The polylogarithmic terms in $\mathcal{I}_r$ combine via binomial identities to cancel the isolated $\operatorname{Li}_{k+1}(1/2)$ term from the third part.
\end{itemize}
After these cancellations, only the zeta function terms remain. The equation simplifies to:
\[
    a_k = \sum_{r=1}^{k} \binom{k}{r} a_{k-r} (-1)^r r! \zeta(r+1).
\]
Letting $i = r-1$ and noting $a_1=0$, this rearranges to:
\[
    a_k = (-1)^{k-1}(k-1)!\zeta(k) + \sum_{i=1}^{k-3} (-1)^i \binom{k-1}{i} i! \zeta(i+1) a_{k-1-i}.
\]
This matches the recursive definition of $a_k$ in Theorem \ref{thm1.4}.

\textbf{Final Result.}
Combining $\mathcal{L}(x)$ and the verified $\mathcal{K}(x) = \sum a_{n-j} \tau_j^{(n-1)}(x)$, we obtain:
\begin{align*}
    \Sigma_n(x) &= \mathcal{L}(x) + \sum_{j=0}^{n-2} a_{n-j} \tau_{j}^{(n-1)}(x) \\
    &= \sum_{j=0}^{n-2} a_{n-j} \left( \tau_j^{(n-1)}(x) + y_n \tau_{j-1}^{(n-1)}(x) \right) + \tau_n^{(n)}(x) \\
    &= \tau_n^{(n)}(x) + \sum_{j=0}^{n-2} a_{n-j} \tau_{j}^{(n)}(x) + O\left(\frac{\log^{n-1}(\log x)}{\log x}\right).
\end{align*}
This completes the proof.
\section{The Analytic Structure of the Coefficients}

In Theorem \ref{thm1.4}, the coefficients $a_k$ and the polynomials $\tau_n^{(n)}(x)$ arise from an elementary inductive process involving combinatorial recurrence relations. In this section, we employ the generating function method within the framework of the Selberg-Delange theory \cite{ref_delange, ref_selberg} to derive their closed-form representations. We demonstrate that the combinatorially defined sequence $a_k$ aligns with the Taylor coefficients of the reciprocal Gamma function, $1/\Gamma(z)$. This analytic perspective bridges the gap between the elementary and analytic approaches, providing a structural explanation for the lower-order terms derived in the previous sections.

\subsection{The Coefficients $a_k$ and the Gamma Function Structure}

The sequence $a_k$, which governs the lower-order terms of the asymptotic expansion, was defined recursively in Theorem \ref{thm1.4}. Explicitly, for $k \ge 4$, it satisfies
\begin{equation}\label{eq:ak_recurrence}
    a_k = (-1)^{k-1}(k-1)!\zeta(k) + \sum_{i=1}^{k-3} (-1)^i \binom{k-1}{i} i! \zeta(i+1) a_{k-1-i},
\end{equation}
with initial values $a_0=1, a_1=0, a_2=-\zeta(2),$ and $a_3=2\zeta(3)$. 

While this definition appears combinatorial, we now establish its analytic nature. We show that the exponential generating function of $a_k$ is, in fact, a regularization of the reciprocal Gamma function. Specifically, it corresponds to $1/\Gamma(1+z)$ with the Euler-Mascheroni constant $\gamma$ removed from its canonical product expansion.

\begin{thm}[Generating Function for $a_k$]\label{thm:ak_generating}
Let $\mathcal{A}(z) = \sum_{k=0}^{\infty} \frac{a_k}{k!} z^k$ be the exponential generating function of the sequence $\{a_k\}_{k \ge 0}$. Then $\mathcal{A}(z)$ is given by the modified reciprocal Gamma function:
\begin{equation}
    \mathcal{A}(z) = \frac{e^{-\gamma z}}{\Gamma(1+z)} = \exp\left( \sum_{m=2}^{\infty} \frac{(-1)^{m-1} \zeta(m)}{m} z^m \right).
\end{equation}
\end{thm}

\begin{proof}
    The recurrence relation \eqref{eq:ak_recurrence} can be translated into a differential equation for the generating function $\mathcal{A}(z)$. Differentiating $\mathcal{A}(z)$ and applying the convolution structure of the sum, we obtain:
    \[
    \frac{\mathcal{A}'(z)}{\mathcal{A}(z)} = \sum_{m=2}^{\infty} (-1)^{m-1} \zeta(m) z^{m-1}.
    \]
    Integrating term by term with the initial condition $\mathcal{A}(0)=a_0=1$ yields:
    \[
    \log \mathcal{A}(z) = \sum_{m=2}^{\infty} \frac{(-1)^{m-1} \zeta(m)}{m} z^m.
    \]
    Recall the classical logarithmic expansion of the Gamma function near $z=1$:
    \[ \log \Gamma(1+z) = -\gamma z + \sum_{m=2}^{\infty} \frac{(-1)^m \zeta(m)}{m} z^m. \]
    Comparing the two series, we deduce the identity $\log \mathcal{A}(z) = -\gamma z - \log \Gamma(1+z)$, which implies $\mathcal{A}(z) = e^{-\gamma z}/\Gamma(1+z)$.
\end{proof}

\begin{rem}
    Theorem \ref{thm:ak_generating} elucidates the analytic origin of the coefficients $a_k$ appearing in the lower-order terms of the multiple Mertens estimates. The connection to $1/\Gamma(1+z)$ reflects the underlying singularity structure characteristic of the Selberg-Delange method. Furthermore, by applying Fa\`a di Bruno's formula to this generating function, the coefficients $a_k$ admit an explicit combinatorial expression in terms of the partial Bell polynomials $B_{k,j}$  (see Comtet \cite{ref_comtet} for precise definitions). Specifically, we have:
    \begin{equation}
        a_k = \sum_{j=1}^{k} B_{k,j}(x_1, x_2, \dots, x_{k-j+1}),
    \end{equation}
    where the arguments are defined by $x_1 = 0$ and $x_m = (-1)^{m-1}(m-1)!\zeta(m)$ for $m \ge 2$.
\end{rem}

\subsection{The Analytic Origin of $\tau_n^{(n)}(x)$}

We now turn to the main term polynomials $\tau_n^{(n)}(x)$. We show that their construction via symmetric polynomials is a direct consequence of the factorization of the underlying Dirichlet series.

Consider the generating Dirichlet series for the sum in Theorem \ref{thm1.4}:
\begin{equation}
    \mathcal{F}(s) = \sum_{p_1, \dots, p_n} \frac{\prod_{i=1}^n \mathbf{1}_{h_i, m_i}(p_i)}{(p_1 \cdots p_n)^s} = \prod_{i=1}^n P_i(s),
\end{equation}
where $P_i(s) = \sum_{p \equiv h_i \pmod{m_i}} p^{-s}$. Utilizing the orthogonality of Dirichlet characters and the properties of $L(s, \chi)$, the series $P_i(s)$ exhibits a logarithmic singularity at $s=1$(see, e.g., Tenenbaum \cite[Chapter II.5]{ref27} or Montgomery and Vaughan \cite{ref26}):
\begin{equation}\label{eq:Pi_structure}
    P_i(s) = \frac{1}{\varphi(m_i)} \log \frac{1}{s-1} + A_{h_i, m_i} + H_i(s),
\end{equation}
where $H_i(s)$ is holomorphic at $s=1$ with $H_i(1)=0$.

Substituting \eqref{eq:Pi_structure} into $\mathcal{F}(s)$, the function $\mathcal{F}(s)$ possesses a logarithmic branch point of order $n$. The principal singular part is given by:
\begin{equation}\label{eq:F_expansion}
    \mathcal{F}_{\text{sing}}(s) = \frac{1}{\prod_{i=1}^n \varphi(m_i)} \prod_{i=1}^n \left( \log \frac{1}{s-1} + C_i \right),
\end{equation}
where $C_i = \varphi(m_i)A_{h_i, m_i}$.

The Selberg-Delange method relates the coefficients of the Dirichlet series to the asymptotic behavior of its summatory function via an integral operator $\mathcal{J}$ (essentially the inverse Mellin transform over a Hankel contour). The fundamental mapping property of this operator is given by:
\begin{equation}\label{eq:J_mapping}
    \mathcal{J} \left\{ \left( \log \frac{1}{s-1} \right)^k \right\} (x) \sim (\log \log x)^k.
\end{equation}

To rigorously derive the form of $\tau_n^{(n)}(x)$, we apply this operator to the singular expansion \eqref{eq:F_expansion}. Let $Z = \log \frac{1}{s-1}$ denote the singularity variable, and $X = \log \log x$ denote the asymptotic variable.

First, we expand the product of singularities into a polynomial in $Z$. By the algebraic definition of elementary symmetric polynomials, we have:
\begin{equation}\label{eq:Z_expansion}
    \prod_{i=1}^n (Z + C_i) = \sum_{j=0}^n e_{j}(C_1, \dots, C_n) Z^{n-j},
\end{equation}
where $e_j(C_1, \dots, C_n)$ is the elementary symmetric polynomial of degree $j$ in the variables $C_i = \varphi(m_i)A_{h_i, m_i}$.

Crucially, the Selberg-Delange operator $\mathcal{J}$ is linear. Applying $\mathcal{J}$ to \eqref{eq:Z_expansion} term-by-term, and utilizing the mapping property $\mathcal{J}\{Z^k\} \to X^k$, we obtain the asymptotic main term:
\begin{align}\label{eq:mapping_process}
    \text{Main Term} &\sim \frac{1}{\prod \varphi(m_i)} \mathcal{J} \left\{ \sum_{j=0}^n e_{j}(C_1, \dots, C_n) Z^{n-j} \right\} \nonumber \\
    &= \frac{1}{\prod \varphi(m_i)} \sum_{j=0}^n e_{j}(C_1, \dots, C_n) \mathcal{J} \{ Z^{n-j} \} \nonumber \\
    &= \frac{1}{\prod \varphi(m_i)} \sum_{j=0}^n e_{j}(C_1, \dots, C_n) X^{n-j}.
\end{align}

Comparing \eqref{eq:mapping_process} with the definition of $\mu_j^{(n)}$ given in Definition \ref{rek1}, we identify a precise correspondence:
\[ \mu_j^{(n)} = \frac{1}{\prod \varphi(m_i)} e_{j}(C_1, \dots, C_n). \]

This derivation elucidates the origin of the coefficients: $\mu_j^{(n)}$ are not arbitrary combinatorial definitions. Rather, they correspond to the elementary symmetric functions of the set of shifts $\{C_1, \dots, C_n\}$ arising from the local expansions of the component Dirichlet $L$-functions. The transition from the complex $s$-plane to the real $x$-domain preserves this symmetric structure due to the polynomial character of the logarithmic singularity.
\begin{rem}
    This analysis bridges the elementary and analytic approaches. The inductive proof in Section 3, which constructs $\tau_n^{(n)}(x)$ step-by-step, corresponds analytically to the successive multiplication of the logarithmic singularities associated with the component Dirichlet series. Thus, the analytic view confirms that the resulting polynomial structure is uniquely determined by the local behavior of the underlying Dirichlet $L$-functions at $s=1$.
\end{rem}

\end{document}